\documentclass[hidelinks,12pt]{amsart}
\usepackage[left=.9truein,right=.9truein,bottom=.9truein,top=.9truein]{geometry}
\usepackage{mathrsfs,amsfonts, amssymb, amsmath, amscd,enumerate,graphicx,tikz,pgfplots,stmaryrd,booktabs,mathtools,young,color,soul,MnSymbol,wasysym,calligra,mathrsfs,hyperref,multicol,float,algpseudocode,algorithm}
\usepackage[balance,small,nohug]{diagrams}
\usepackage{multirow} 
\usepackage{comment}

\pgfplotsset{compat=1.14}
\usetikzlibrary{shapes,plotmarks,matrix,arrows}
\hypersetup{pdfborder = {0 0 0}}

\newcommand{\sheafhom}{\mathscr{H}\text{\kern -3pt {\calligra\large om}}\,}
\newcommand{\sheafext}{\mathscr{E}\text{\kern -3pt {\calligra\large xt}}\,}

\newcommand{\Aut}{\textnormal{Aut}}

\newtheorem{theo}{Theorem}[section]
\newtheorem{theorem}[theo]{Theorem}
\newtheorem*{theorem*}{Theorem}
\newtheorem{proposition}[theo]{Proposition}
\newtheorem*{proposition*}{Proposition}

\newtheorem*{lemma*}{Lemma}

\newtheorem*{corollary*}{Corollary}

\newtheorem*{conjecture*}{Conjecture}
\theoremstyle{definition}
\newtheorem{definition}[theo]{Definition}
\newtheorem*{definition*}{Definition}
\newtheorem{example}[theo]{Example}
\newtheorem*{example*}{Example}

\newtheorem{problem}[theo]{Problem}
\newtheorem*{problem*}{Problem}
\newtheorem{notation}[theo]{Notation}
\newtheorem{remark}[theo]{Remark}
\newtheorem*{remark*}{Remark}

\algnewcommand\algorithmicforeach{\textbf{for each}}
\algdef{S}[FOR]{ForEach}[1]{\algorithmicforeach\ #1\ \algorithmicdo}

\usepackage{float}

\begin{document}
\title[Maximal skew sets of lines on a Hermitian surface]{Maximal skew sets of lines on a Hermitian surface and a modified Bron-Kerbosch algorithm}
\author{Anna Brosowsky, Haoyu Du, Madhav Krishna, Sandra Nair, Janet Page, and Tim Ryan}
\address{ }
\email{ }
\maketitle

\section{Introduction}
The goal of this paper is to study maximal sets of skew lines on certain algebraic surfaces, and motivated by this study, to propose a variant of the clique listing problem from computer science and combinatorics and an algorithm which solves that problem.

Let $k$ be a field extension of $\mathbb{F}_{q^2}$ where $q = p^e$ for a prime $p$, and define the Fermat surface of degree $q+1$ to be the surface $X$ in $\mathbb{P}^3$ defined by the vanishing of the polynomial $x^{q+1}+y^{q+1}+z^{q+1}+w^{q+1}$, i.e.,
\[X = \mathbb{V}\left(x^{q+1}+y^{q+1}+z^{q+1}+w^{q+1}\right).\]
These surfaces, and the surfaces projectively equivalent to them, have long been objects of study in algebraic geometry, commutative algebra, and combinatorics.
Over $k = \mathbb{F}_{q^2}$, the surfaces projectively equivalent to a Fermat surface of degree $q+1$ are known as \emph{Hermitian surfaces}; these have been extensively studied due to the Frobenius automorphism, $[x:y:z:w] \mapsto [x^q:y^q:z^q:w^q]$, being an involution akin to complex conjugation, see \cite{Segre,BC,Hirschfeld} for some of the foundational work.
More recently, over any field of characteristic $p >0$, these have been studied because they fit into a larger class of surfaces called \emph{extremal surfaces}, which were introduced in \cite{extremal} because their affine cones have ``most singular'' possible cone points among all surfaces of the same degree (as measured by an invariant known as the $F$-pure threshold).  Over an algebraically closed field, all \emph{smooth} extremal surfaces are projectively equivalent to the Fermat surface.

One fascinating facet of the geometry of these surfaces is the presence of many lines on them despite their arbitrarily high degree. From an algebraic geometry viewpoint, these lines are interesting as they exceed the bound of lines possible on a surface of a given degree in characteristic zero \cite{Segre,BauerandRams} and conjecturally, may provide
the maximum number of lines possible on a surface of a given degree in characteristic $p > 0$. 
From a combinatorial viewpoint, these lines are interesting as they are a geometric realization of the \emph{Hermitian generalized quadrangle} $\mathcal{H}(3,q^2)$, which is a particular incidence structure that is a prototypical example of a generalized quadrangle. 
These viewpoints converge in their interest in maximal sets of skew lines on the surface.
In algebraic geometry, these generalize the classical constructions of six skew lines on a smooth cubic surface, while in combinatorics these have been extensively studied as ``maximal partial spreads'', e.g., \cite{ACE,BC,Cossidente,CossidentePavese2,EbertHirschfeld,GLSV,MetschStorme}.
Even the most straightforward questions about these maximal sets of skew lines are unknown: \textit{What are the possible sizes of maximal skew sets of lines on a Hermitian surface $X$ of degree $q+1$, and how many maximal skew sets are there of each size?}
This is equivalent to asking what are the sizes of maximal partial spreads in the Hermitian generalized quadrangle $\mathcal{H}(3,q^2)$ and how many of them are there.
In this paper, we use new methods to list all sizes for maximal skew sets of lines on $X$ when $q=2,3,$ and $4$, which were previously known. Further, we count the number of these maximal skew sets when $q=3$.
\begin{proposition*}
The number of maximal sets of skew lines of each size on $X$ is known for $q=3$.
\end{proposition*}  
\noindent We also give the count in the previously known case of $q=2$. 
See Section \ref{sec:computational-results} for the exact counts as well as partial results for $q=4$.
Most excitingly in this work, working on these results led us to a novel question in graph theory and computer science, and an algorithm to solve it.

This connection arose by first converting our original question about lines on surfaces to a question in graph theory.
In particular, we note that counting maximal skew sets of lines on $X$ is equivalent to counting maximal complete subgraphs (cliques) of a specific graph (the complement of the incidence graph of lines on $X$), so that our problem is equivalent to a special case of the \textit{clique listing problem}.
The clique listing problem asks for an algorithm to list all cliques of any given graph.
Solving the clique listing problem for all graphs is difficult as even deciding if there is a clique of a fixed size $k$ is an NP-complete problem \cite{Karp}. In this work, we propose a modified version of this problem that may be easier to approach. 
Intuitively, graphs with the most cliques will have many automorphisms, and many of the cliques will be in the same orbit under the action of the automorphism group. Motivated by this intuition, a different way to package the information of all cliques in a graph is a set of generators for the automorphism group and a single representative of each orbit; it would also suffice to use any subgroup of the automorphism group.
Computationally, this approach will be most useful when the automorphism group of the graph (or at least a subgroup of it) is known, which motivates the following question.
\begin{problem*}[Clique Orbit Listing Problem]
Given a graph $G$ and a group
acting on it, list (at least) one representative of each orbit of maximal cliques in $G$.
\end{problem*}
We give a modified version of the Bron-Kerbosch algorithm with pivoting which solves this problem, see Algorithm~\ref{alg:bk-modtest}.
Our algorithm has the disadvantages of often giving more than one representative of each orbit and depending on the order in which the vertices are listed, but has the big advantage of being very fast in practice.
In particular, the original Bron-Kerbosch algorithm with pivoting would have taken prohibitively long to run on our computer resources for the graph corresponding to the generalized quadrangle $\mathcal{H}(3,4^2)$, but our algorithm was able to run relatively quickly.  
In future work, both the theoretical and practical advantages of this algorithm, and other algorithms solving the clique orbit listing problem, should be explored. 

Finally, inspired by our computational results, we give an explicit construction of a skew set of size $\frac{3}{2}q^2-\frac{1}{2}q+1$ on any Hermitian surface of degree $q+1$.  While skew sets of this size were previously known in other language \cite{ACE2, Coolsaet}, we believe that our construction can shed new light on their geometry.  
Furthermore, we find a new lower bound on the number of maximal skew sets of at least this size, thereby showing that there are at least 
$\binom{q+1}{3}2^{(q^2 -q + 2)/2}$ 
maximal cliques in $\mathcal{H}(3,q^2)$.

The paper is organized as follows: Section~\ref{sec:background} provides the necessary background on Hermitian varieties and generalized quadrangles, Section~\ref{sec:clique-conversion} discusses the conversion of the geometric problem of finding maximal sets of skew lines to finding cliques in graphs and algorithms to solve the clique listing problems, Section~\ref{sec:computational-results} gives our computational results on Hermitian generalized quadrangles, and Section~\ref{sec:large-example} gives a geometric construction of a large set of skew lines found during our computations.

\subsection{Acknowledgements}
Anna Brosowsky was partially supported by NSF grants 1952399, 1840234, and 2101075.  
Tim Ryan was partially supported by Karen Smith's Keeler and Fulton Professorships.  
We would also like to thank Karen Smith for many useful conversations.

\section{Background}
\label{sec:background}
\subsection{Smooth extremal surfaces}
Let $p$ be a prime and let $k$ be a field extension of $\mathbb{F}_{q^2}$ where $q = p^e$.  
We are studying lines on the \emph{Fermat surface of degree $q+1$}, which is the projective variety 
\[X:= \mathbb{V}(x^{q+1} + y^{q+1} + z^{q+1} + w^{q+1}) \subseteq \mathbb{P}_{k}^3.\]
Fix $X$ as this surface throughout the paper.
Over any such $k$, this $X$ is an example of a \emph{smooth extremal surface}.  
When $k$ is algebraically closed, $X$ is projectively equivalent to $\mathbb{V}(\sum_{i,j=1}^4 a_{ij}x_i^qx_j)$ for any choice of $a_{ij}$ such that the matrix with entries $a_{ij}$ has full rank, which in particular means that $X$ is projectively equivalent to any smooth extremal surface of this degree \cite{beauville,extremal}.

We recall several facts about $X$.  In the case that $k = \mathbb{F}_{q^2}$, these are well known facts about Hermitian varieties, see for example \cite{BC,Hirschfeld}.  The case that $k$ is any other field containing $\mathbb{F}_{q^2}$ follows from \cite{geometry,Katsura}.

\begin{theorem}\label{factsaboutX}
Let $X=\mathbb{V}(x^{q+1} + y^{q+1} + z^{q+1} + w^{q+1})$.  Then
\begin{enumerate}
    \item $X$ contains $q^4 + q^3 + q + 1$ lines.
    \item At every point where two lines on $X$ intersect, $q+1$ lines on $X$ intersect.  We call these points \textbf{star points}\footnote{When $k = \mathbb{F}_{q^2}$, all points on $X$ are star points.}.
    \item $X$ contains $q^5 + q^3 + q^2 +1$ star points.
    \item Each line on $X$ contains $q^2 + 1$ star points.
    \item The lines on $X$ do not form any triangles\footnote{In other words, any three lines in the same plane on $X$ all intersect at the same point.}.
\end{enumerate}
\end{theorem}

Further, the lines and star points on $X$ (which is all of $X$ in the case that $k = \mathbb{F}_{q^2}$) form an incidence structure called a \emph{generalized quadrangle}.

\begin{definition}
A \emph{generalized quadrangle} $GQ(s,t)$ is a collection of lines and points such that
\begin{enumerate}
    \item On every line there are exactly $s + 1$ points. There is at most one point on two distinct lines.
    \item Through every  point there are exactly $t + 1$ lines. There is at most one line through two distinct points.
    \item For every point $p_1$ not on a line $L$, there is a unique line $M$ and a unique point $p_2$, such that $p_1$ is on $M$, and $p_2$ on $M$ and $L$.  (This implies $GQ(s,t)$ does not contain triangles.)
\end{enumerate}
\end{definition}

The following proposition is well known in the case that $X$ is Hermitian (i.e., when $k = \mathbb{F}_{q^2}$) \cite{Hirschfeld,BC}, but holds for any field $k$ containing $\mathbb{F}_{q^2}$.

\begin{proposition}
The lines and star points on $X$ form a generalized quadrangle with $s= q^2$ and $t = q$, often called the \emph{Hermitian generalized quadrangle}. 
\end{proposition}

\begin{proof}
To see this, note that properties (1) and (2) follow from properties (4) and (2) of Theorem~\ref{factsaboutX}.  The uniqueness of $M$ in property (3) is property (5) in Theorem~\ref{factsaboutX}.  Finally, to see the existence of $M$ in property (3), note that if $p_1$ is a star point on $X$ then $T_{p_1}X$, the tangent plane to $X$ at $p_1$, intersects $X$ in a collection of $q+1$ lines.  In $\mathbb{P}_{k}^3$, every line intersects every plane, so that $L \subseteq X$ must intersect $T_{p_1}X$, necessarily on one of the $q+1$ lines of $T_{p_1}X \cap X$.  This intersection point is a star point since it is at the intersection of two lines by Proposition~8.12 in \cite{extremal} (and the fact that $X$ is smooth).
\end{proof}

In some literature, this generalized quadrangle  is denoted by $\mathcal{H}(3,q^2)$ and called a \emph{Hermitian generalized quadrangle}.  Here, $3$ refers to the fact that $X \subseteq \mathbb{P}^3$.   
This Hermitian generalized quadrangle has long been an object of study in finite geometry and combinatorics.
It also has applications to coding theory, see \cite{EtzionStorme} for an introduction to this connection.
For a more thorough introduction to the study of Hermitian quadrangles, see \cite{Segre,PayneThas,Thas}. 
In the language of this field, a smooth extremal variety is ``the set of all absolute points of a non-degenerate unitary polarity'', see for example \cite{CossidentePavese}; they are examples of the more general notion of a ``polar space'' \cite{Hirschfeld}.
A line on a smooth extremal surface is a ``generator''.
A (maximal) set of $t$ skew lines is a ``(complete) $t$-span'' or a ``(maximal) partial spread''.
A ``spread'' is a maximal partial spread containing all of the star points of the extremal surface.
It is known that no spread exists in $\mathcal{H}(3,q^2)$ \cite{Segre}, and so the attention is naturally turned to complete partial spreads, i.e., maximal sets of skew lines.
In particular, what is the smallest and largest size of maximal partial spreads \cite{EbertHirschfeld}?
In the language of algebraic geometry, what is the smallest and largest size of a maximal set of skew lines on a smooth extremal surface?

Throughout this paper, we focus on finding maximal skew sets of lines on $\mathcal{H}(3,q^2)$,
or in other words on $X = \mathbb{V}\left(x^{q+1}+y^{q+1}+z^{q+1}+w^{q+1}\right)$.  In order to do computations, we enumerate the lines on $X$ explicitly.  Throughout this paper, we will use the following notation:
\begin{notation}\label{not:indices}
Let $\mu$ be a primitive element of $\mathbb{F}_{q^2}$, and let $\nu = \mu^{\frac{(q-1)\mathrm{gcd}(2,q)}{2}}$. This will ensure that $\mu^{q^2 - 1} = 1$ and that $\nu^{q+1} = -1$.
Let $\mu^{a_{1}}$, $\dots$, $\mu^{a_{(q-2)(q+1)}}$ be the powers of $\mu$ that are not powers of $\nu$ (listed so that $0 \leq a_i < a_{i+1} \leq q^2-1$ for each $i$), and let $\mu^{a_{i,1}}$, $\dots$, $\mu^{a_{i,q+1}}$ be the $(q+1)^{st}$ roots of $-1-\mu^{(q+1)a_i}$  (again, listed so that $0 \leq a_{i,j} < a_{i,j+1} \leq q^2-1$ for each $i,j$).
We enumerate the lines of $X$ in the following way.
\begin{itemize}
    \item $L_{i(q+1) + j} = \mathbb{V}(x + \nu^{2i+1}y,\ z + \nu^{2j+1}w)$ for $0 \leq i, j \leq q$
    \item $L_{(q+1)^2 + i(q+1) + j} = \mathbb{V}(x + \nu^{2i+1}z,\ y + \nu^{2j+1}w)$ for $0 \leq i, j \leq q$
    \item $L_{2(q+1)^2+ i(q+1) + j} = \mathbb{V}(x + \nu^{2i+1}w,\ y + \nu^{2j+1}z)$ for $0 \leq i, j \leq q$
    \item $L_{(3+i)(q+1)^2 + j(q+1) + k} = \mathbb{V}(-x + \mu^{a_i}y + \mu^{a_{ij}}w,\ -\mu^{qa_i}x -y + \mu^{a_{ik}}z)$ for $0 \leq i\leq q^2-q-2$ and $0 \leq j, k\leq q$
\end{itemize}
\end{notation}

\section{Conversion to the clique problem and a modification of the BK-algorithm}
\label{sec:clique-conversion}

We first convert the problem of finding maximal skew sets of lines to another well known problem.

\begin{definition}
Let $q$ be a power of a prime $p$, and let $X = \mathbb{V}(x^{q+1} + y^{q+1} + z^{q+1} + w^{q+1})$.  Denote by $G(X)$ the graph whose vertices correspond to lines on $X$, and whose vertices are connected by an edge if and only if the corresponding lines are skew.
\end{definition}

Finding maximal skew sets of lines on $X$ is equivalent to finding maximal complete subgraphs (cliques) on $G(X)$.  Thus, we are interested in solving the following problem for the particular graph $G = G(X)$.

\begin{problem}[Clique (Listing) Problem]
Given a graph $G$, list all the maximal cliques in $G$.
\end{problem}

\begin{remark}
We note that equivalently, we could have considered the \emph{incidence graph} of the lines in $X$, i.e., the graph where each vertex represents a line on X, and an edge connects two vertices if and only if the corresponding lines intersect. This incidence graph is the complement of $G(X)$ and finding a maximal skew set of lines on $X$ is equivalent to finding a maximal ``independent set," or set of vertices with no edges between them, on this incidence graph.  In general, the clique problem for a graph $G$ is equivalent to listing all maximal independent sets on the complement of $G$.  
\end{remark}

It is known that a simpler variant of the clique listing problem is NP-complete \cite{Karp}.  
One of the most efficient algorithms to solve the clique problem is the Bron-Kerbosch algorithm with pivoting.
In particular, it has the optimal worst-case running time (without printing), $O(3^{\frac{n}{3}})$, since any $n$-vertex graph has at most $3^{\frac{n}{3}}$ maximal cliques \cite{MR182577,TOMITA200628}.

In Algorithm~\ref{alg:bk-pivot}, we give the pseudo-code for the Bron-Kerbosch algorithm with pivots (or BK-algorithm), where $N(v)$ denotes the neighbors of (i.e., set of vertices sharing an edge with) the vertex $v$ \cite{BK}.


\begin{figure}[H]
\begin{algorithm}[H]
  \caption{Bron-Kerbosch algorithm with pivoting
    \label{alg:bk-pivot}}
  \begin{algorithmic}[1]
    \Statex{{\bf inputs:} }
            \Statex{ $R$: a set of vertices to Require} 
            \Statex{ $P$: a set of Potential vertices to consider }
            \Statex{$E$: a set of vertices to Exclude}
    \Statex
    \Function{BronKerbosch}{$R,P,E$}
     \If{$P$ and $E$ are both empty}
      \State report $R$ as maximal clique
      \EndIf
     \State choose a pivot vertex $u$ in $P \cup E$
     \ForEach{vertex $v$ in $P \setminus N(u)$}
        \State BronKerbosch($R \cup \{v\}, P \cap N(v), E \cap N(v)$)
        \State $P \coloneqq P \setminus \{v\}$
        \State $E \coloneqq E \cup \{v\}$
     \EndFor
    \EndFunction
  \end{algorithmic}
\end{algorithm}
\end{figure}

In brief, the original Bron-Kerbosch algorithm systematically extends a starting clique (normally the empty clique) in every possible way. 
The addition of pivoting is a pruning technique which significantly speeds up the algorithm. 
At each step, pivoting chooses a vertex $u$ which is not already in the clique.  Note that either the vertex $u$ will be in a maximal clique or a vertex disconnected from $u$ will be in a maximal clique (otherwise, if only neighbors of $u$ were in a clique, we could add $u$ and make it larger).  
Thus, it suffices to extend the clique by considering vertices that are not neighbors of $u$ (including $u$ itself).

With the algorithm above on the graph $G(X)$, we can enumerate the maximal skew sets of lines on $X$ when $q=2$ and $q=3$.  However, it becomes unfeasible for higher degrees.  We can push this computation further by using the automorphisms of our surface $X$.

\subsection{The BK-algorithm with orbits}

In order to improve the BK-algorithm, we can use the fact that automorphisms of our surface $X$ preserve sets of skew lines. 
Rephrased in terms of our graph $G(X)$, we note that any automorphism of $G(X)$ will take any clique to another clique of the same size.  Instead of listing every clique in $G(X)$, we could list one representative from each orbit of cliques in $G(X)$, under the action of any group $\mathcal{G}$ on $G(X)$.

\begin{problem}[Clique Orbit Listing Problem]
Given a graph $G$ and a group acting on it, list (at least) one representative of each orbit of maximal cliques in $G$.
\end{problem}

\begin{example}
With notation as in Notation~\ref{not:indices}, we have that the sets $S_1 := \{L_0, L_4, L_8, L_{10}, L_{12}\}$ and $S_2 := \{L_0, L_4, L_8, L_{11}, L_{15}\}$ are both skew sets of lines of size five on $X$ when $q = 2$.  Then the automorphism $\varphi$ induced by the map $k[x,y,z,w] \to k[x,y,z,w]$ sending 
\[
x \mapsto \nu^2z,\  y \mapsto \nu^2w,\  z \mapsto x, \mbox{ and }w \mapsto y
\]
will take $S_1$ to $S_2$. It fixes $L_0$ , $L_4$, and $L_8$, while it sends $L_{10} \mapsto L_{11}$ and $L_{12} \mapsto L_{15}$. So these two sets are in the same orbit of skew sets under the action of the automorphism group.
\end{example}

Note that solving this problem also solves the Clique Problem.  
The entire list of maximal cliques can be recovered from a list of orbit representatives by acting by the group. 
In particular, listing the orbits and the (generators of the) group can be a more efficient way to list the cliques.
\begin{example}
\label{ex: worst case}

In \cite{MR182577}, it is proved that (up to isomorphism) there is a unique graph which has the maximal number of cliques on $n=3k$ vertices, which is $3^{\frac{n}{3}}$ cliques.
This graph has $n=3k$ vertices, labelled $p_{3i+j}$ where $0\leq i \leq k-1$ and $0\leq j \leq 2$. There is an edge between $p_{3i+j}$ and $p_{3k+m}$ if and only if $i\neq k$.
The automorphism group of this graph has order $k!\cdot 6^k$ (as you can permute each set of three vertices $\{p_{3i},p_{3i+1},p_{3i+2}\}$ and the three vertices within each group).
In fact, the graph has a single maximal clique up to automorphism, so the clique orbit listing problem is far simpler in this case as you can list a single orbit representative and the $4$ generators of the automorphism group. 
Similar results hold for unique graphs with the maximal number of cliques when $n=3k+1$ and $n=3k+2$.
\end{example}
Intuitively, the existence of a high number of cliques in a graph with a fixed number of vertices will force high levels of symmetry in the graph, i.e., the graph would have a lot of automorphisms. The graphs which take a longer time for the clique listing problem will be those with a higher number of nontrivial orbits of cliques. These are precisely the graphs for which the clique orbit listing problem should be the biggest improvement on the clique listing problem.

Below, we give an algorithm which provides at least one representative for each orbit of cliques in a graph.
Note, this modified Bron-Kerbosch algorithm does not list a unique representative for each orbit, and in the case that a graph has no automorphisms, the Bron-Kerbosch algorithm with orbits reduces exactly to the traditional Bron-Kerbosch algorithm.
We give our modified Bron-Kerbosch algorithm with orbits in pseudo-code below, see Algorithm~\ref{alg:bk-modtest}.

\begin{figure}[htb] 
\begin{algorithm}[H]
  \caption{Modified Bron-Kerbosch algorithm with orbits \label{alg:bk-modtest}}  
  \begin{algorithmic}[1]
    \Statex{{\bf inputs:}}
            \Statex{ $R$: a set of vertices to Require} 
            \Statex{ $P$: a set of Potential vertices to consider}
            \Statex{$E$: a set of vertices to Exclude}
            \Statex{StabR: the stabilizer of $R$ as an ordered list}
            
    \Statex
    \Function{BronKerbosch}{$R,P,E$, StabR}
        \If{$P$ and E are both empty}
            \State{report R as maximal clique}
        \EndIf
        \State{Choose a pivot vertex $u$ in $P \cup E$}
        \State{PCopy $\coloneqq P$}
        \State{OrbitReps $\coloneqq ()$}
        \While{PCopy$\neq ()$}
        \State{Append the first element $L$ of PCopy to OrbitReps.}
        \For{$g \in$ StabR}
        \State{Remove $g(L)$ from PCopy}
        \EndFor
        \EndWhile
     \ForEach{vertex $v$ in $\mathrm{OrbitReps}$}
      \State{StabRv $:=\{\phi \in$ StabR $\mid \phi(v) = v\}$}
      \State BronKerbosch($R \cup \{v\}, P \cap N(v), E \cap N(v)$, StabRv)
      \State $P \coloneqq P \setminus \{v\}$
      \State $E \coloneqq E \cup \{v\}$
      \State StabR $\coloneqq$ StabRv
     \EndFor
    \EndFunction
  \end{algorithmic}
\end{algorithm}
\end{figure}

\subsection{Applying the algorithm}
To effectively implement this algorithm, we first use the following theorem to reduce the starting case.
\begin{theorem} [\cite{geometry}, Theorem 5.1.2]
The automorphism group of $X$, denoted $\Aut(X)$, acts transitively on sets of three skew lines.  In particular, given any two skew sets of three lines $\{\ell_1, \ell_2, \ell_3\}$ and $\{m_1,m_2,m_3\}$ there is an automorphism $\phi$ taking the ordered set $\{\ell_1, \ell_2, \ell_3\}$ to the ordered set $\{m_1,m_2,m_3\}$.
\end{theorem}

Thus, we may assume without loss of generality that our skew set contains a fixed (ordered) set of three lines.
As per Notation~\ref{not:indices} we use $\left(L_0,L_{q+2},L_{2q+4}\right)$ as our initial list of three skew lines, and we search only for skew sets containing $\left(L_0,L_{q+2},L_{2q+4}\right)$.
Because of this, we will consider the action by $\mathcal{G} \subseteq \Aut(X)$ where $\mathcal{G}$ is the subgroup of $\Aut(X)$ which sends the ordered list $\left(L_0,L_{q+2},L_{2q+4}\right)$ to itself, typically called the \emph{(ordered) stabilizer} of $\left(L_0,L_{q+2}, L_{2q+4}\right)$.
In particular, we start with $R = \left(L_0,L_{q+2},L_{2q+4}\right)$, $P = \{\text{common neighbors of }\left(L_0,L_{q+2},L_{2q+4}\right)$ in $G(X) \}$, $E = \varnothing$, and $\mathrm{StabR} =\text{ the stabilizer of } \left(L_0,L_{q+2},L_{2q+4}\right)$.

\section{Computational Results}
\label{sec:computational-results}

We were able to enumerate all maximal skew sets of lines on $X$ for $q = 2,3$ using the usual BK-algorithm, and were able to find at least one representative of each orbit for $q = 4$ using our modified BK-algorithm with orbits.

\subsection{Maximal sets of skew lines on a smooth extremal cubic surface}
As the most classically studied case, everything was previously known when $q=2$. 
We include this case for the sake of completeness. 
There are only maximal skew sets of size 6, of which there are 72, and size 5, of which there are 216.
The existence of the 72 maximal skew sets of size 6 on any smooth cubic surface over any algebraically closed field is a classical fact whose study starts with Schl\"afli's construction of double sixes \cite{Schlafli}.
The existence of 216 maximal skew sets of size was also known \cite{EbertHirschfeld}.

\bigskip
\begin{center}
\begin{tabular}{|c|c|c|}
    \hline
    \multicolumn{3}{|c|}{$q=2$}\\
    \hline
     $n$ & 5 & 6 \\ \hline
     number of skew sets of size $n$ & 216 & 72 \\ \hline
\end{tabular}
\end{center}

\subsection{Maximal sets of skew lines on a smooth extremal quartic surface}
\label{subsec:quartic}

In the case of $q=3$, Ebert and Hirschfeld \cite{EbertHirschfeld} showed that $16$ was the largest possible size of any maximal skew set and provided a skew set of size 16. 
They further showed that the orbit of this skew set under the action of the automorphism group of the surface has size 2268, but were unable to show whether or not there were any other orbits of skew sets of size 16. 
Our count proves their orbit is the only orbit of size 16.  In the same paper, they prove that any maximal skew set must have size at least $2q+1$ for any $q$ (and at least $2q+2$ for $q>3$), but they did not provide a maximal skew set of size $2\cdot 3+1 = 7$ in the case $q=3$.  We show the following:

\begin{proposition}\label{quarticskewsets}
The sizes of maximal skew sets on a smooth extremal quartic surface are $7$, $10$, $11$, $12$, $13$, and $16$.
Further, there exists a unique orbit of maximal skew sets of size $16$.
\end{proposition}
This result follows from our computations, done by running the classical
BK-algorithm with pivots on the complement of the incidence graph of lines on the smooth extremal quartic surface.
Further, the algorithm counts the number of maximal skew sets of each size, which is summarized in the table below.

\bigskip

\begin{center}
\begin{tabular}{|c|c|c|c|c|c|c|}
    \hline
    \multicolumn{7}{|c|}{$q=3$}\\
    \hline
     $n$ & 7 & 10 & 11 & 12 & 13 & 16 \\ \hline
     number of skew sets of size $n$ & 5,184 & 766,584 & 3,447,360 & 816,480 & 181,440 & 2,268 \\ \hline
\end{tabular}
\end{center}

\smallskip

That these were the only sizes of maximal skew sets was obtained earlier in \cite{Cimrakova} using different methods. 
The uniqueness of the maximal orbit of size 16 and the exact count of the number of maximal skew sets of each size are however both new results.

\subsection{Maximal sets of skew lines on a smooth extremal quintic surface}

In the case of $q=4$, while the BK-algorithm with pivots was unfeasible to run on our computers, we were able to run our modified BK-algorithm with orbits. 
Similar to the previous case, this result was also first obtained by computer in \cite{Cimrakova}: these are the only possible sizes of maximal skew sets in this graph.

\begin{proposition}
The sizes of maximal skew sets on a smooth extremal quintic surface are $13$, $15$, $16$, $17$, $18$, $19$, $20$, $21$, $22$, $23$, $24$, and $25$.
\end{proposition}

This result follows from running the modified BK-algorithm with orbits for the complement of the incidence graph of lines on the smooth extremal quartic surface. In this case, we provide at least one representative of each orbit of maximal skew sets.  Since our algorithm produces multiple representatives of some orbits, the number of representatives we count (listed below) give an upper bound on the number of orbits of maximal skew sets of each size, and a lower bound on the number of skew sets of each size.
\begin{center}
\begin{tabular}{|c|c|c|c|c|c|c|}
    \hline
    \multicolumn{7}{|c|}{$q=4$}\\
    \hline
     $n$ & 13 & 14 & 15 & 16 & 17 & 18 \\ \hline
     \# representatives of skew sets we produce 
     & 393 & 0 & 8,404 & 250,322 & 11,814,949 & 61,083,983 \\ \hline
\end{tabular}
\end{center}
\begin{center}
\begin{tabular}{|c|c|c|c|c|c|c|}
    \hline
     19 & 20 & 21 & 22 & 23 & 24 & 25 \\ \hline
     119,190,408 & 149,850,860 & 100,366,159 & 31,916,219 & 3,546,647 & 23,507 & 21,082  \\ \hline
\end{tabular}
\end{center}

\section{A large skew set}
\label{sec:large-example}
After studying our computational results in detail, we were also able to construct a large skew set.
A similar construction can be found outlined in other language in \cite{ACE2} where it is attributed to Thas, and it is worked out in detail for odd $q$ in the dual setting in \cite{Coolsaet}.  However, we believe our description of this construction sheds light on the geometry of these large skew sets.  In particular, Thas' skew set is equivalent to the one we give in Theorem~\ref{largeskew}, but in our language it is clear there are many other similar large skew sets, which we mention in Proposition~\ref{countingourconstruction}.  Further, we use our construction to give an exponential lower bound on the number of these skew sets, which is new.  
Additionally, this provides a lower bound on the number of maximal skew sets as we show that no two of these skew sets can be contained in the same maximal skew set.

\begin{theorem}\label{largeskew}
There exist maximal sets of skew lines of size at least $\frac{3}{2}q^2-\frac{1}{2}q+1$ on a smooth extremal surface $X$ of degree $q+1$.
\end{theorem}

\begin{proof}
We will first construct the set of lines of that size, and then show that it is skew set.

Fix three skew lines $m_1$, $m_2$, and $m_3$.
Note, these will not be in the skew set we construct.
These three lines define a unique quadric surface $Q$ \cite[2.12]{Harris}.  By Theorem~5.0.4 of \cite{geometry}, $Q\cap X$ consists of $2q + 2$ lines $\{m_1, \dots, m_{q+1}, \ell_1, \dots \ell_{q+1}\}$ where each $m_i$ is skew to each $m_j$ and intersects each $\ell_k$ and similarly for the $\ell_i$'s.

By Proposition~5.2.2 of \cite{geometry}, there are $q(q-1)$ star chords\footnote{Here, by \emph{star chord} we mean a line which is not on $X$ but which contains $q + 1$ star points of $X$ (in this case, there will be one star point at each intersection $c_i \cap m_j$ and $c_i' \cap m_j$). In the case that $X$ is a Hermitian variety, these are sometimes called Baer sublines \cite[6.2.1]{Hirschfeld98}.} in the ruling of $Q$ which contains $\ell_1, \dots, \ell_{q+1}$.  
By Corollary~5.2.5 of \cite{geometry}, these are in split into $\frac{1}{2}q(q-1)$ pairs,
\[
\left\{\{c_1,c_1'\},\dots,\{c_{\frac{1}{2}q(q-1)},c_{\frac{1}{2}q(q-1)}'\}\right\}
\]
of dual star chords\footnote{Two star chords are dual if and only if they have the same set of $(q+1)^2$ lines through their respective star points which is true if and only if they have at least three lines on the surface in common between their sets of $(q+1)^2$ lines on the surface.}. 
Let $p_{i,j} = m_i \cap c_j$ and $p_{i,j}' = m_i\cap c_j'$.  By Lemma~5.2.1 of \cite{geometry}, these are all star points.
Denote the lines $\overline{p_{1,j}p_{2,j}'}$, $\overline{p_{2,j}p_{3,j}'}$, and $\overline{p_{3,j}p_{1,j}'}$ by $\ell_{q+3j-1}$, $\ell_{q+3j}$, and $\ell_{q+3j+1}$, respectively.  Note, these indices start at $q+2$ and go through $q+\frac{3}{2}q(q-1)+1 = \frac{3}{2}q^2-\frac{1}{2}q+1$.
By Corollary~5.2.5 of \cite{geometry}, these are each lines on $X$.
\begin{center}
\includegraphics{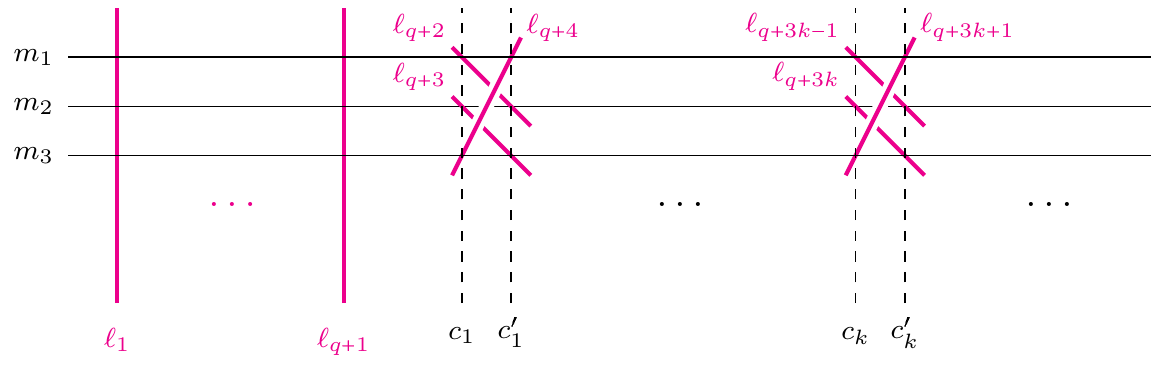}
\end{center}

To conclude the proof, it suffices to show that the set $\{\ell_1,\dots,\ell_{\frac{3}{2}q^2-\frac{1}{2}q+1}\}$ is a skew set.  By construction, each line in $\{\ell_1, \dots, \ell_{q+1}\}$ intersects all three of $m_1,m_2,$ and $m_3$, and each line in $\{\ell_{q+2}, \dots, \ell_{\frac{3}{2}q^2-\frac{1}{2}q+1}\}$ intersects exactly two of the lines $m_1,m_2$ and $m_3$.  Thus, for any pair $\ell_i,\ell_j$ (where $1 \leq i,j \leq \frac{3}{2}q^2-\frac{1}{2}q+1$), we must have that there is some $m_k$ (where $1\leq k \leq 3$) such that $\ell_i$ and $\ell_j$ both intersect $m_k$.  This shows that $\ell_i$ and $\ell_j$ cannot intersect since if they did $\{\ell_i,\ell_j,m_k\}$ would form a triangle.
\end{proof}

In fact, we can show the following lower bound (which is likely very far from sharp):
\begin{proposition}\label{countingourconstruction}
There are at least $\frac{(q+1)q(q-1)}{3} \cdot 2^{\frac{1}{2}q(q-1)}$ maximal skew sets of size at least $\frac{3}{2}q^2-\frac{1}{2}q+1$ on $X =\mathbb{V}(x^{q+1} + y^{q+1} + z^{q+1} + w^{q+1})$.
\end{proposition}
\begin{proof}

Fix a quadric configuration $\mathcal{Q}$ on $X$; i.e., an intersection of a quadric $Q$ with $X$, where $Q \cap X$ is a set of $2q+2$ lines $\{\ell_1, \dots, \ell_{q+1}, m_1, \dots, m_{q+1}\}$, where the $\ell_i$'s are all skew, the $m_i$'s are all skew, and each $\ell_i$ intersects each $m_j$. Fix one of the rulings on $\mathcal{Q}$.  There are two ways to do this.  Once we have fixed a ruling, there are $\binom{q+1}{3}$ options to fix three lines of this ruling, say $m_1$, $m_2$, and $m_3$.  Using the notation above, we have $\frac{1}{2}q(q-1)$ pairs of dual star chords, and for each of these we have two choices of three skew lines: 
$\{\overline{p_{1,j}p_{2,j}'},\,  \overline{p_{2,j}p_{3,j}'},\, \overline{p_{3,j}p_{1,j}'}\}$ or   
$\{\overline{p_{1,j}p_{3,j}'},\, \overline{p_{2,j}p_{1,j}'},\,\overline{p_{3,j}p_{2,j}'}\}$.  Then $\{\ell_1, \dots, \ell_{q+1}\}$ along with a choice of triple for each of the $\frac{1}{2}q(q-1)$ dual star chord pairs gives a skew set by the argument above.  

Finally, note that no two of these skew sets formed as above from a fixed quadric configuration $\mathcal{Q}$ can lie together in a larger skew set.  To see this, first note that if they came from different rulings on $\mathcal{Q}$, then one set would contain $\ell_1, \dots \ell_{q+1}$ and the other would contain $m_1,\dots, m_{q+1}$, which intersect.  
 
If they came from the same ruling but were distinct, then for some $i$, one set will contain $\overline{p_{1,j}p_{2,j}'}$ while the other will contain $\overline{p_{1,j}p_{3,j}'}$. These lines are not skew, since they intersect in $p_{1,i}$.

Thus, we have $2 \cdot \binom{q+1}{3} \cdot 2^{\frac{1}{2}q(q-1)}$ distinct maximal skew sets coming from this quadric configuration (our configurations may not be maximal, but at least each lie in distinct maximal skew sets).
\end{proof}

Since we do not know exactly how they interact, we only count the sets of skew lines coming from a single quadric configuration in Proposition~\ref{countingourconstruction}, but there are likely many distinct skew sets coming from other quadric configurations (and there are $\frac{1}{2}(q^3+1)(q^2+1)q^4$ quadric configurations on $X$ \cite[Corollary~5.1.3]{geometry}).  In the case of $q=2$, all $72$ skew sets of size $\frac{3}{2}q^2-\frac{1}{2}q+1 = 6$ are of this form.  In fact, any skew set of size $6$ will show up as a skew set as constructed above in $\binom{6}{3} = 20$ different quadric configurations (any quadric configuration given by the quadric determined from any $3$ of the six lines in the skew set).  When we multiply our count $\frac{(q+1)q(q-1)}{3} \cdot 2^{\frac{1}{2}q(q-1)} = 4$ above by the $\frac{1}{2}(q^3+1)(q^2+1)q^4 = 360$ quadric configurations, we get $1440$, which is $20$ times the number of maximal skew sets of $6$ lines.

In the case of $q=3$, we have also computationally checked that all maximal skew sets of size $\frac{3}{2}q^2-\frac{1}{2}q+1 = 13$ also come from the construction above.  In this case, each skew set of size $13$ can be constructed in this way using $4$ different quadric configurations.  We have
$$
\left(\frac{(q+1)q(q-1)}{3} \cdot 2^{\frac{1}{2}q(q-1)}\right) \cdot \left(\frac{1}{2}(q^3+1)(q^2+1)q^4\right) = 64 \cdot 11340 = 725760 = 4 \cdot 18144,
$$
so that our bound times the number of quadric configurations gives four times the number of maximal sets of skew lines of size $13$.

For $q= 4$, we computationally verified that skew sets of size $\frac{3}{2}q^2-\frac{1}{2}q+1 = 23$ in the construction above each come from only one quadric configuration, so that we have at least $$\left(\frac{(q+1)q(q-1)}{3} \cdot 2^{\frac{1}{2}q(q-1)}\right) \cdot \left(\frac{1}{2}(q^3+1)(q^2+1)q^4\right) = 181043200$$
skew sets of size $23$ in this case.  These are not maximal, but instead can be uniquely extended to maximal skew sets by adding one line.  We expect that in general, for higher $q$, each of these skew sets should come from only one quadric configuration.  If this is true, then for $q \geq 4$, we would have the much larger lower bound of $\left(\frac{(q+1)q(q-1)}{3} \cdot 2^{\frac{1}{2}q(q-1)}\right) \cdot \left(\frac{1}{2}(q^3+1)(q^2+1)q^4\right)$ for the number of maximal skew sets of size at least $\frac{3}{2}q^2-\frac{1}{2}q+1$.

\bibliographystyle{alpha}
\bibliography{Support_Files/Citations}
\end{document}